\documentclass[a4paper, 11pt]{amsart}
\usepackage{amsfonts, amsmath, amssymb, amsthm,mathtools, url,dsfont}
\usepackage{graphicx}
\usepackage[english] {babel}
\usepackage{enumitem}
\usepackage[draft]{fixme}
\usepackage{array}
\usepackage{cellspace}
\newtheorem{thm}{Theorem}

\newtheorem{prop}[thm]{Proposition}

\theoremstyle{definition}

\newcommand{\ve}{\varepsilon}

\DeclareMathOperator{\disc}{disc}

\usepackage{subcaption}
\usepackage{graphicx}

\begin{document}

\begin{abstract}
It is well-known that for every $N \geq 1$ and $d \geq 1$ there exist point sets $x_1, \dots, x_N \in [0,1]^d$ whose discrepancy with respect to the Lebesgue measure is of order at most $(\log N)^{d-1} N^{-1}$. In a more general setting, the first author proved together with Josef Dick that for any normalized measure $\mu$ on $[0,1]^d$ there exist points $x_1, \dots, x_N$ whose discrepancy with respect to $\mu$ is of order at most $(\log N)^{(3d+1)/2} N^{-1}$. The proof used methods from combinatorial mathematics, and in particular a result of Banaszczyk on balancings of vectors. In the present note we use a version of the so-called transference principle together with recent results on the discrepancy of red-blue colorings to show that for any $\mu$ there even exist points having discrepancy of order at most $(\log N)^{d-\frac12} N^{-1}$, which is almost as good as the discrepancy bound in the case of the Lebesgue measure.
\end{abstract}

\title[{Tusn\'ady's problem, transference, low-discrepancy points}]{Tusn\'ady's problem, the transference principle, and non-uniform QMC sampling}

\author{Christoph Aistleitner} 
\address{Christoph Aistleitner, Institute of Analysis and Number Theory, TU Graz, Austria}
\email{christoph.aistleitner@jku.at}

\author{Dmitriy Bilyk} 
\address{Dmitriy Bilyk, School of Mathematics, University of Minnesota, USA}
\email{dbilyk@math.umn.edu}

\author{Aleksandar Nikolov} 
\address{Aleksandar Nikolov, Department of Computer Science, University of Toronto}
\email{anikolov@cs.toronto.edu}

\thanks{The first author is supported by the Austrian Science Fund (FWF), project Y-901.}
\thanks{The third author is supported by an NSERC Discovery Grant.}

\maketitle

\section{Introduction and statement of results}

Many problems from applied mathematics require the calculation or estimation of the expected value of a function depending on several random variables; such problems are for example the calculation of the fair price of a financial derivative and the calculation of the expected loss of an insurance risk. The expected value $\mathbb{E}\big(g(Y^{(1)}, \dots, Y^{(d)} )\big)$ can be written as
\begin{equation} \label{ABNgenint}
\int_{\mathbb{R}^d} g \big(y^{(1)}, \dots, y^{(d)} \big) ~d\nu\big(y^{(1)}, \dots, y^{(d)}\big),
\end{equation}
where $\nu$ is an appropriate probability measure describing the joint distribution of the random variables $Y^{(1)}, \dots, Y^{(d)}$. Since a precise calculation of the value of such an integral is usually not possible, one looks for a numerical approximation instead. Two numerical methods for such problems are the \emph{Monte Carlo method} (MC, using random sampling points) and the \emph{Quasi-Monte Carlo method} (QMC, using cleverly chosen deterministic sampling points), where the QMC method is often preferred due to a faster convergence rate and deterministic error bounds. However, in the QMC literature it is often assumed that the problem asking for the value of \eqref{ABNgenint} has at the outset already been transformed into the problem asking for the value of 
\begin{equation} \label{ABNspecint}
\int_{[0,1]^d} f \big(x^{(1)}, \dots, x^{(d)} \big) ~dx^{(1)}\cdots x^{(d)}.
\end{equation}
Thus it is assumed that the integration domain is shrunk from $\mathbb{R}^d$ to $[0,1]^d$, and that the integration measure is changed from $\nu$ to the Lebesgue measure. In principle, such a transformation always exists -- this is similar to the way how general multivariate random sampling is reduced to sampling from the multivariate uniform distribution, using for example the Rosenblatt transform \cite{ABNrosen} (which employs sequential conditional inverse functions). The shrinking procedure is less critical, albeit also non-trivial\footnote{One may wish to avoid this shrinking, and carry out QMC integration directly on $\mathbb{R}^d$ instead. One way of doing so is to ``lift'' a set of sampling points from $[0,1]^d$ to $\mathbb{R}^d$ rather than shrinking the domain of the function. See for example \cite{ABNdilp} and the references there.}; however, the change from a general measure to the uniform measure is highly problematic, in particular because QMC error bounds depend strongly on the regularity of the integrand. Note that the change from $\nu$ to the Lebesgue measure induces a transformation of the original function $g$ to a new function $f$, and this transformation may totally ruin all ``nice'' properties of the initial function $g$ such as the existence of derivatives or the property of having bounded variation. This is not the place to discuss this topic in detail; we just note that the main problem is not that each of $Y^{(1)}, \dots, Y^{(d)}$ may have a non-uniform marginal distribution, but rather that there may be a strong dependence in the joint distribution of these random variables (which by Sklar's theorem may be encoded in a so-called \emph{copula} -- see \cite{ABNnelsen} for details). We refer to \cite{ABNcambou} for a discussion of these issues from a practitioner's point of view.\\

There are two possible ways to deal with the problems mentioned in the previous section, which can be seen as two sides of the same coin. Either one tries to transform the points of a classical QMC point set in such a way that they can be used for integration with respect to a different, general measure, and such that one obtains a bound for the integration error in terms of the discrepancy (with respect to the uniform measure) of the original sequence -- this is the approach of Hlawka and M\"uck \cite{ABNhla1,ABNhla2}, which has been taken up by several authors. On the other hand, one may try to sample QMC points directly in such a way that they have small discrepancy with respect to a given measure $\mu$, and use error bounds which apply in this situation. This is the approach discussed in \cite{ABNai1,ABNai2}, where it is shown that both key ingredients for QMC integration are given also in the setting of a general measure $\mu$: there exist (almost) low-discrepancy point sets, and there exists a Koksma--Hlawka inequality giving bounds for the integration error in terms of the discrepancy (with respect to $\mu$) of the sampling points and the degree of regularity of the integrand function. More precisely, let $\mu$ be a normalized Borel measure on $[0,1]^d$, and let 
\begin{equation}\label{ABNd1}
D_N^*(x_1, \dots, x_N; \mu) = \sup_{A \in \mathcal{A}} \left| \frac{1}{N} \sum_{n=1}^N \mathds{1}_A (x_n) - \mu(A) \right|
\end{equation}
be the star-discrepancy of $x_1, \dots, x_N \in [0,1]^d$ with respect to $\mu$; in this definition $\mathds{1}_A$ denotes the indicator function of $A$, and the supremum is extended over the class $\mathcal{A}$ of all axis-parallel boxes contained in $[0,1]^d$ which have one vertex at the origin. Then in \cite{ABNai1} it is shown that for arbitrary $\mu$ there exist points $x_1, \dots, x_N \in [0,1]^d$ such that
\begin{equation} \label{ABNdnmu}
D_N^*(x_1, \dots, x_N; \mu) \leq 63 \sqrt{d} \frac{(2 + \log_2 N)^{(3d+1)/2}}{N}.
\end{equation}
This improved an earlier result of Beck \cite{ABNbeck}, where the exponent of the logarithmic term was $2d$. Note the amazing fact that the right-hand side of \eqref{ABNdnmu} does not depend on $\mu$ (whereas the choice of the points $x_1, \dots, x_N$ clearly does). The estimate in \eqref{ABNdnmu} should be compared with corresponding results for the case of the uniform measure, where it is known that there exist point sets whose discrepancy is of order at most $(\log N)^{d-1} N^{-1}$ (so-called \emph{low-discrepancy point sets}).\footnote{While preparing the final version of this manuscript we learned that already in 1989 J\'ozsef Beck \cite[Theorem 1.2]{ABN_ober} proved a version of \eqref{ABNdnmu} with the stronger upper bound $\mathcal{O} \left(N^{-1} (\log N)^{d+2} \right)$. (However, the implied constant was not specified in his result.)}\\

A Koksma--Hlawka inequality for general measures was first proved by G\"otz \cite{ABNgotz} (see also \cite{ABNai2}). For any points $x_1, \dots, x_N$, any normalized measure $\mu$ on $[0,1]^d$ and any $d$-variate function $f$ whose variation $\textup{Var}~f$ on $[0,1]^d$ (in the sense of Hardy--Krause)  is bounded, we have
\begin{equation} \label{ABNKH}
\left|\int_{[0,1]^d} f(x) d\mu (x)  - \frac{1}{N} \sum_{n=1}^N f(x_n) \right| \leq D_N^*(x_1, \dots, x_N; \mu) ~\textup{Var}~f. 
\end{equation}
This is a perfect analogue of the original Koksma--Hlawka inequality for the uniform measure. Combining \eqref{ABNdnmu} and \eqref{ABNKH} we see that in principle QMC integration is possible for the numerical approximation of integrals of the form
$$
\int_{[0,1]^d} f \big(x^{(1)}, \dots, x^{(d)} \big) d\mu(x^{(1)}\dots x^{(d)}),
$$
and that the convergence rate is almost as good as in the classical setting for the uniform measure. However, it should be noted that for the case of the uniform measure many explicit constructions of low-discrepancy point sets are known (see \cite{ABNdpd}), whereas the proof of \eqref{ABNdnmu} is a pure existence result and it is totally unclear how (and if) such point sets can be constructed with reasonable effort.\\

The purpose of the present note is to show that for any $\mu$ there actually exist points $x_1, \dots, x_N$ whose discrepancy with respect to $\mu$ is of order at most $(\log N)^{d-\frac12}N^{-1}$; this is quite remarkable, since it exceeds the corresponding bound for the case of the uniform measure only by a factor $(\log N)^{\frac12}$. 

\begin{thm} \label{ABNth1}
For every $d \geq 1$ there exists a constant $c_d$ (depending only on $d$) such that the following holds. For every $N \geq 2$ and every normalized Borel measure on $[0,1]^d$ there exits points $x_1, \dots, x_N \in [0,1]^d$ such that
$$
D_N^*(x_1, \dots, x_N; \mu) \leq c_d \frac{(\log N)^{d-\frac12}}{N}.
$$
\end{thm}

The  proof of this theorem uses a version of the so-called transference principle, which connects the combinatorial and geometric discrepancy, see Theorem \ref{ABNtransfer}. The novelty and the main observation of this paper lies in the fact that this principle is still valid for general measures $\mu$. This observation was previously made by Matou{\v{s}}ek in \cite{ABNmat2}, without providing any details, but otherwise it has been largely overlooked. In addition we shall use  new upper bounds for Tusn\'ady's problem due to the third author  (for the discussion of Tusn\'ady's problem, see \S\ref{ABNtusnady}). If one wants to refrain from the application of unpublished results,  one can use Larsen's \cite{ABNlarsen} upper bounds for Tusn\'ady's problem instead, which yield Theorem \ref{ABNth1} with the exponent $d+1/2$ instead of $d- 1/2$ of the logarithmic term (see also \cite{ABNmat}). An exposition of the connection between geometric and  combinatorial discrepancies, together with the proof of Theorem \ref{ABNth1}, is given in the following section.\\

Before turning to combinatorial discrepancy, we want to make several remarks concerning Theorem \ref{ABNth1}. Firstly, whereas the conclusion of Theorem \ref{ABNth1} is stated for finite point sets, one can use a well-known method to find an infinite sequence $(x_n)_{n \geq 1}$ whose discrepancy is of order at most $(\log N)^{d+1/2} N^{-1}$ for all $N \geq 1$. A proof can be modeled after the proof of Theorem 2 in \cite{ABNai1}.\\

Secondly, while the upper bound in Theorem \ref{ABNth1} is already very close to the corresponding upper bound in the case of the classical discrepancy for the uniform measure, there is still a gap. One wonders whether this gap is a consequence of a deficiency of our method of proof, or whether the discrepancy bound in the case of general measures really has to be larger than that in the case of the uniform measure. In other words, we have the following open problem, which we state in a slightly sloppy formulation.\\

{\bf Open problem:}  %When thinking about the existence of low-discrepancy point sets as the problem of approximating measures by finite atomic measures of equal weights: 
Is the largest upper bound for the smallest possible discrepancy the one for the case of Lebesgue measure? In other words, 
is there any measure which is more difficult to approximate by finite atomic measures with equal weights than Lebesgue measure?\\

We think that this is a problem of significant theoretical interest. Note, however, that even in the classical case of the Lebesgue measure the problem asking for the minimal order of the discrepancy of point sets is famously still open\footnote{This is known as the \emph{Great Open Problem} of discrepancy theory, see, e.g., \cite{ABNBC}, page 8, and \cite{ABNmatbook}, page 178.}; while in the upper bounds for the best known constructions the logarithmic term has exponent $d-1$, in the best known lower bound the exponent is $(d-1)/2+\varepsilon_d$ for small positive $\varepsilon_d$ (see \cite{ABNbil2} for the latter result, and \cite{ABNbil1} for a survey). It is clear that some measures are much easier to approximate than Lebesgue measure -- think of the measure having full mass at a single point, which can be trivially approximated with discrepancy zero. On the other hand, if the measure has a non-vanishing continuous component then one can carry over the orthogonal functions method of Roth \cite{ABNroth} -- this is done in \cite{ABNchen}. However, the lower bounds for the discrepancy which one can obtain in this way are the same as those for the uniform case (and not larger ones). Intuitively, it is tempting to assume that the Lebesgue measure is essentially extremal -- simply because intuition suggests that a measure which is difficult to approximate should be ``spread out everywhere'', and should be ``equally large'' throughout the cube.\\

\section{Combinatorial discrepancy  and the transference principle}
\label{ABNsect:comb-disc}

\subsection{Combinatorial discrepancy}
Let $V = \{1, \dots, n\}$ be a ground set and $\mathcal{C} = \{C_1, \dots, C_m\}$ be a system of subsets of $V$. Then the (combinatorial) discrepancy of $\mathcal{C}$ is defined as
$$
\disc \mathcal{C} = \min_{y \in \{-1,1\}^V } \disc (\mathcal{C},y),
$$
where the vector $y$ is called a \emph{(red-blue) coloring} of $V$ and 
$$
\disc (\mathcal{C} ,y) = \max_{i \in \{1, \dots, m\}} \left| \sum_{j \in C_i} y_j \right|
$$
is the discrepancy of the coloring $y$. We may visualize the entries of the vector $y$ as representing two different colors (usually red and blue). Then $\disc (\mathcal{C},y)$ is the maximal difference between the number of red and blue elements of $V$ contained in a test set $C_i$, and $\disc \mathcal{C}$ is the minimal value of $\disc (\mathcal{C},y)$ over all possible colorings $y$. For more information on this combinatorial notion of discrepancy, see \cite{ABNchaz,ABNmatbook}. 

We will only be concerned with  a geometric variation of this notion, i.e. the case when $V$ is a  point set in $[0,1]^d$, and when the set system $\mathcal{C}$ is the collection of all sets of the form ${{G}} \cap V$,  where $G\in \mathcal G$ and ${\mathcal{G}}$ is some geometric  collection of subsets of $[0,1]^d$: standard choices include, e.g., balls, convex sets, boxes (axis-parallel or arbitrarily rotated) etc.  Let $\disc(N,{\mathcal{G}},d)$ denote the maximal possible discrepancy in this setting; that is, set
$$
\disc(N,\mathcal{G},d) = \max_{P} (\disc \mathcal{C}),
$$
where the maximum is taken over all sets $P$ of $N$ points in $[0,1]^d$ and where $\mathcal{C} = \mathcal{G} \vert_P = \left\{ {G} \cap P :\, G \in \mathcal G \right\}$.  \\

\subsection{Tusn\'{a}dy's problem}\label{ABNtusnady} We denote by  $\mathcal{A}$ (as in the previous section)  the class of all axis-parallel boxes having one vertex at the origin. The problem of finding sharp   upper and lower bounds for $\disc(N,\mathcal{A},d)$ as a function of $N$ and $d$ is known as  \emph{Tusn\'ady's problem}. For the history and background of the problem, see \cite{ABNbeck2,ABNmat2,ABNmat}.\\

We will use the following result, which has been recently announced by the third author~\cite{ABNNikolov17}.  A slightly weaker result~\cite{ABNBansalG16} (with exponent $d$) by Bansal and Garg has been presented at MCQMC 2016 (and is also still unpublished). A yet weaker, but already published result (with exponent $d+1/2$ instead of $d-1/2$) is contained  in \cite{ABNlarsen}.

\begin{prop} \label{ABNprop}
For every $d \geq 1$ there exists a constant $c_d$ (depending only on $d$) such that for all $N \geq 2$
$$
\disc(N,\mathcal{A},d) \leq c_d (\log N)^{d-\frac12}.
$$ 
\end{prop}

Finally, we want to note that Tusn\'ady's problem is still unsolved;  the known (and conjecturally optimal)  lower bounds are of the  order $(\log N)^{d-1}$, see \cite{ABNmat}. As will become clear from the next subsection,  further improvements of the upper bounds for Tusn\'ady's problem would directly imply improved upper bounds in Theorem \ref{ABNth1}. Note that actually Tusn\'ady's problem also falls within the framework of ``discrepancy with respect to a general measure $\mu$''. Given $N$ points, let $\mu$ be the discrete measure that assigns mass $1/N$ to each of these points. Then Tusn\'ady's problem asks for a set $P$ of roughly $N/2$ points such that the discrepancy of $P$ with respect to $\mu$ is small -- under the additional requirement that the elements of $P$ are chosen from the original set of $N$ points. This additional requirement also explains why lower bounds for Tusn\'ady's problem do not imply lower bounds for the problem discussed in the present paper.

\subsection{Transference principle} 
It is known that upper bounds for the combinatorial discrepancy of red-blue colorings can be turned into upper bound for the smallest possible discrepancy of a point set in the unit cube. This relation is called the \emph{transference principle}.\\

For a system $\mathcal G$ of measurable subsets of $[0,1]^d$, its (unnormalized)  geometric discrepancy is defined  as  
\begin{equation}\label{ABNd2}
D_N (\mathcal G) = \inf_{P: \, \# P = N}  \sup_{G \in \mathcal{G}} \left|  \sum_{p \in P}  \mathds{1}_G (p) - N \lambda(G) \right|,
\end{equation}
where the infimum is taken over all $N$-point sets $P \subset [0,1]^d$ and $\lambda(G)$ is the Lebesgue measure of $G$. The \emph{transference principle}, roughly speaking,  says that (under some mild assumptions on the collection $\mathcal G$) {\bf{\emph{the geometric discrepancy is bounded  above by the combinatorial discrepancy,}}} i.e. $ D_N (\mathcal G) \lesssim \disc(N,\mathcal{G},d)$ with the symbol  ``$\lesssim$" interpreted loosely. Thus upper bounds on combinatorial discrepancy yield upper bounds for its geometric counterpart.  This relation, in general,  cannot be reversed: in the case when $\mathcal G$ is the collection of all convex subsets of the unit cube,  we have that $ D_N (\mathcal G)$ is of the order $N^{1-\frac{2}{d+1}}$, while $ \disc(N,\mathcal{G},d)$ is of the order $N$ as $N\rightarrow \infty$ (see, e.g., \cite{ABNmatbook}).\\

In \cite{ABNmatbook,ABNmat} it is mentioned that M.~Simmonovits attributes the idea of  this principle to Vera T. S\'os. It was used in the context of Tusn\'ady's problem by Beck \cite{ABNbeck2} in 1981, and is stated in a rather general form in \cite{ABNlovasz}. It can be found also in Matou{\v{s}}ek's book \cite[p.~20]{ABNmatbook} and in \cite{ABNmat}. In all these instances it is formulated in a version which bounds the geometric  discrepancy of point sets \emph{with respect to the Lebesgue measure}  (as defined in \eqref{ABNd2}) in terms of the combinatorial discrepancy of red-blue colorings.\\  %For example, the transference principle is formulated in roughly the following form in \cite{ABNmat}.

However, upon examination of the proof, it turns out that the argument carries over to the case of the discrepancy with respect to an arbitrary measure (the only significant requirement is that the measure allows an $\varepsilon$-approximation for arbitrary $\ve$; see below).  Similar to \eqref{ABNd1} and \eqref{ABNd2},  we define the geometric discrepancy of $\mathcal G$ with respect to a Borel measure $\mu$ as 
\begin{equation}\label{ABNd3}
D_N (\mathcal G,\mu) = \inf_{P: \, \# P = N}  D (P,\mathcal G,,\mu),
\end{equation}{ where }  
\begin{equation} D (P,\mathcal G,\mu) =   \sup_{G \in \mathcal{G}} \left|  \sum_{p \in P}  \mathds{1}_G (p) - N \mu (G) \right|.
\end{equation}
 Note that with this definition $D_N^*(x_1, \dots, x_N; \mu)$, as defined in \eqref{ABNd1}, satisfies $D_N^*(x_1, \dots, x_N; \mu) = \frac1{N} D \big( \{ x_1,\dots,x_N\}, \mathcal A, \mu \big)$. Below we state and prove the transference principle in  a rather general form  for  arbitrary measures. The statement is similar to that in \cite{ABNmatbook,ABNmat} in the case of Lebesgue measure, and the proof follows along the lines of \cite{ABNmatbook}. 
%Below we give the proof of Theorem \ref{ABNth1}, in a form which is similar to the proof of the transference principle given in \cite{ABNmatbook}.\\

\begin{thm}[{\bf{Transference principle for general measures}}] \label{ABNtransfer}
Let $\mu$ be a Borel probability measure on $[0,1]^d$ and let  $\mathcal{G}$ be a class of Borel subsets of $[0,1]^d$ such that $[0,1]^d \in \mathcal{G}$. Suppose that  
\begin{equation}\label{ABNdecay}
D_N (\mathcal G, \mu)  =  o (N) \,\,\, \textup{ as }\,\,\, N \rightarrow \infty.
\end{equation} 
% the smallest possible discrepancy (with respect to the Lebesgue measure and with respect to $\mathcal{C}$) of a set of $N$ points tends to zero as $N \to \infty$. 
Assume furthermore that  the combinatorial discrepancy of $\mathcal{G}$ satisfies  
\begin{equation}
\disc(N,\mathcal{G},d) \leq h(N)
\end{equation}
 for some function $h$ with the property that  $h(2N) \leq (2-\delta)h(N)$ for some fixed $\delta>0$. Then for every $N$ \begin{equation}\label{ABNcdbound}
 D_N (\mathcal G, \mu) =  \mathcal{O}(h(N)),
 \end{equation}  i.e.  there exist points $x_1,\dots,x_N \in [0,1]^d$ % whose discrepancy with respect to the measure $\mu$ and the family  $\mathcal{G}$ is at most 
 so that $ D \big( \{x_1,\dots, x_N\},\mathcal G,\mu\big) \le C h(N)$.
\end{thm}

\begin{proof} Set $\ve = h(N)/N$, and using \eqref{ABNdecay} choose a positive integer  $k$  so large that there exists a set $P_0$ of $2^k N$ points in $[0,1]^d$ %whose discrepancy with respect to $\mu$ is at most $\ve$. 
so that  $ \frac{D (P_0,\mathcal G,\mu)}{2^k N} \le \ve$. 
%Denote this point set by $P_0$. The fact that such a point set exists follows from \eqref{ABNdnmu}. 
%By Proposition \ref{ABNprop} we can find a coloring of 
By \eqref{ABNcdbound} we can find a red-blue coloring of  the  set $P_0$  with discrepancy at most $h (2^k N)$. The difference between the total number of red and blue points  is  also at most $h (2^k N)$, since the full unit cube is an element of our class of test sets. Without loss of generality we may assume that there are no more red than blue points (otherwise switch the roles of the red and the blue points). We keep all the red points and only so many of the blue points as to make sure  that in total we have half the number of the original points, while we dispose of all the other blue points. Write $P_1$ for the new set. The cardinality of $P_1$ is $2^{k-1} N$. Furthermore, %the discrepancy with respect to $\mu$ of the point set $P_1$ is at most 
$$
 \frac{D (P_1,\mathcal G,\mu)}{2^{k-1} N} \le  \ve + {\frac{h(2^k N)}{2^{k-1} N}}.
$$
To see why this is the case, note that an arbitrary set $G \in \mathcal{G}$ contains between $2^k N \mu(G) - \ve 2^k N$ and $2^k N \mu(G) + \ve 2
^k N$ elements of $P_0$. Thus it contains between
$$
\frac12 \big( {2^k N \mu(G) - \ve 2^k N} - h (2^k N) \big) \quad \textrm{and} \quad \frac12 \big( {2^k N \mu(G) + \ve 2^k N} + h (2^k N) \big)
$$
red elements of $P_0$, and consequently between 
$$
\frac12 \big( {2^k N \mu(G) - \ve 2^k N} - h (2^k N) \big) \quad \textrm{and} \quad \frac12 \big( {2^k N \mu(G) + \ve 2^k N} + 2h (2^k N) \big)
$$
elements of $P_1$, where the upper bound is increased by $h(2^k N)/2$ since we had to add at most so many blue points in order to make sure that $P_1$ has the desired cardinality. % (note that the discrepancy with respect to $\mu$ is divided by the total number of points for normalization, while the combinatorial discrepancy is not). 
Repeating this procedure, we obtain a point set $P_2$ of cardinality $2^{k-2} N$ whose discrepancy with respect to $\mu$ is at most
$$
\frac{D (P_2,\mathcal G,\mu)}{2^{k-2} N} \le \ve + {\frac{h(2^k N)}{2^{k-1} N}} + {\frac{h(2^{k-1} N)}{2^{k-2} N}}.
$$
We repeat this procedure over and over again, until we arrive at a point set $P_k$ which has cardinality $N$, and whose discrepancy  with respect to $\mu$ satisfies 
$$
 D(P_k, \mathcal G, \mu ) \le  \underbrace{\varepsilon N}_{=h(N)}  + \sum_{j=0}^k \frac{h (2^{k-j} N )}{2^{k-j-1}  } ~\leq~ C h (N),
$$
where we have used the condition that $h(2N) \leq (2-\delta)h(N)$. Note that the value of $C$ does not depend on the measure $\mu$. This finishes  the proof.
\end{proof}

\subsection{Proof of Theorem \ref{ABNth1}} Theorem \ref{ABNth1} now easily follows from the combinatorial discrepancy estimate in Tusn\'ady's problem (Proposition \ref{ABNprop})  and the transference principle for general measures (Theorem \ref{ABNtransfer}) applied in the case $\mathcal G = \mathcal A$.  The only point that needs checking is whether $\mu$ satisfies  the approximability condition \eqref{ABNdecay} with respect to the collection $\mathcal A$ of axis-parallel boxes with one vertex at the origin. But \eqref{ABNdecay} follows trivially from the prior result \eqref{ABNdnmu}.\\

A slightly more direct way to prove \eqref{ABNdecay} would be to observe that the collection $\mathcal A$ is a VC class (its VC-dimension is $d$), see e.g. \cite{ABNchaz} for definitions. This implies that \cite{ABNVC} it is a uniform  Glivenko--Cantelli class, i.e. $$ \sup_\mu \mathbf E \sup_{A \in \mathcal{A}} \left| \frac{1}{N} \sum_{n=1}^N \mathds{1}_A (x_n) - \mu(A) \right|  \rightarrow 0 \,\, \textup{ as } \,\, N \rightarrow \infty,$$ where the expectation is taken over independent random points $x_1$,...,$x_N$  with distribution $\mu$, and the supremum is over all probability measures $\mu$ on $[0,1]^d$. This immediately yields \eqref{ABNdecay}.\\

In conclusion we want to make some remarks on possible algorithmic ways of finding a point set satisfying the conclusion of Theorem \ref{ABNth1}. Following the proof of the transference principle, two steps are necessary. First one has to find the $\ve$-approximation of $\mu$. The existence of such a point set is guaranteed in the proof as a consequence of \eqref{ABNdnmu}. However, this is not of much practical use since the point set for \eqref{ABNdnmu} cannot be constructed explicitly. However, in Corollary 1 of \cite{ABNai2} it is proved that a set of $2^{26} d \ve^{-2}$ random points which are sampled randomly and independently from the distribution $\mu$ will have, with positive probability, a discrepancy with respect to $\mu$ which is less than a given $\ve$. This result is deduced from large deviation inequalities, and thus for a larger value of the constant (say $2^{35}$ instead of $2^{26}$) the probability that the random point set has discrepancy at most $\ve$ with respect to $\mu$ will be extremely close to 1. We can think of $\ve$ as roughly $1/N$; for the cardinality of the random point set, this would give roughly $2^{35} d N^2$. Note, however, that in each iteration of the coloring procedure the number of points is halved; accordingly, starting with $2^{35} d N^2$ points leads to a point set of cardinality $N$ after a rather small number of steps. Admittedly, by using random points for the $\ve$-approximation the whole problem is only shifted rather than solved; it is typically rather difficult to draw independent random samples from a general multivariate distribution $\mu$.\\

The second part of the proof (the coloring procedure) is less of a problem from an algorithmic point of view, since in recent years much work has been done on algorithms for actually finding balanced colorings in combinatorial discrepancy theory. In particular, the recent bound~\cite{ABNBansalG16} for Tusn\'ady's problem due to Bansal and Garg mentioned in Section~\ref{ABNsect:comb-disc} is algorithmic in the following sense: they describe an efficient randomized algorithm that, given a set $V$ of $N$ points in $[0,1]^d$, finds a red-blue coloring of $V$ which has discrepancy at most $c_d(\log N)^d$ with probability arbitrarily close to $1$. ``Efficient'' here means that the running time of the algorithm is bounded by a polynomial in $N$. In another recent preprint, Levy, Ramadas and Rothvoss~\cite{ABNLevyRR16} describe an efficient deterministic algorithm that achieves the same guarantees as the randomized algorithm from~\cite{ABNBansalDG16} for the Koml\'os problem. Since the techniques used in~\cite{ABNBansalG16} are very closely related to those of~\cite{ABNBansalDG16}, it seems likely that an efficient deterministic algorithm to find colorings for Tusn\'ady's problem with discrepancy bounded by $c_d(\log N)^d$ can be constructed via the methods of~\cite{ABNLevyRR16}. Unfortunately, the stronger bound of $c_d(\log N)^{d-\frac12}$ proved in~\cite{ABNNikolov17} relies on an existential result of Banaszczyk~\cite{ABNBana13}, and no efficient algorithm is currently known that achieves this bound.

\section*{Acknowledgements}

This paper was conceived while the authors were taking a walk in the vicinity of Rodin's \emph{Gates of Hell} sculpture on Standford University campus during the MCQMC 2016 conference. We want to thank the MCQMC organizers for bringing us together. Based on this episode, we like to call the open problem stated in the first section of this paper the \emph{Gates of Hell Problem}.

\end{document}